\documentclass{article}
\usepackage{amsmath,amssymb,amsthm,amsfonts,amscd}

 \newtheorem{thm}{Theorem}[section]
 \newtheorem{cor}[thm]{Corollary}
 
 \newtheorem{prop}[thm]{Proposition}
 \theoremstyle{definition}
 \newtheorem{defn}[thm]{Definition}
 \theoremstyle{remark}
 \newtheorem{rem}[thm]{Remark}
 \newtheorem*{ex}{Example}
 \numberwithin{equation}{section}

\newcommand{\cw}{\stackrel{D}{\rightharpoonup}}

\newcommand{\N}{\mathbb{N}}
\newcommand{\R}{\mathbb{R}}

\newcommand{\Z}{\mathbb{Z}}

\begin{document}
\title{Concentration analysis and cocompactness}

\author{Cyril Tintarev\\
Department of Mathematics\\
Uppsala University\\
Box 480\\
751 06 Uppsala\\
Sweden}
\maketitle

\begin{abstract} Loss of compactness that occurs in may significant PDE settings can be expressed in a well-structured
form of profile decomposition for sequences. Profile decompositions are formulated in relation to a triplet $(X,Y,D)$, 
where $X$ and $Y$ are Banach spaces, $X\hookrightarrow Y$, and $D$ is, typically, a set of surjective isometries on both $X$ and $Y$. 
A profile decomposition is a representation of a bounded sequence in $X$ as a sum of elementary concentrations of the form 
$g_kw$, $g_k\in D$, $w\in X$, and a remainder that vanishes in $Y$. A necessary requirement for $Y$ is, therefore,  
that any sequence in $X$ that develops no $D$-concentrations has a subsequence convergent in the norm of $Y$. 
An imbedding $X\hookrightarrow Y$ with this property is called $D$-cocompact, a property weaker than, but related to, compactness. 
We survey known cocompact imbeddings and their role in profile decompositions. 
\end{abstract}


\section{Introduction}

Convergence of functional sequences is easy to obtain in problems where
one can invoke compactness, typically via a compact imbedding $X\hookrightarrow Y$
of two Banach spaces. In many cases, however, one deals with problems
in functional spaces that possess some non-compact invariance, such as
translational or scaling invariance, which produces non-compact orbits 
and thus makes any invariant imbedding trivially non-compact. Fortunately, the same set of invariances that destroys
compactness can be employed to restore it. This approach, historically
called the concentration compactness principle, emerged in the 1980's from the analysis of concentration
phenomena by Uhlenbeck, Brezis, Coron, Nirenberg, Aubin and Lions, and remains widely used, in a standardized formulation,
in terms of sequences of measures, due to Willem and Chabrowski. The early concentration analysis has been then refounded 
by two systematic theories developed in mutual isolation, a functional-analytic one
(\cite{Revista}, with origins in \cite{Solimini}), and a wavelet-based one, in function spaces, 
(see Bahouri, Cohen and Koch \cite{BCK} whose origins are in G\'erard's paper \cite{Gerard}). 

The purpose of this survey is to present current results of general concentration analysis as well 
as some areas of its advanced applications, such as elliptic problems of Trudinger-Moser
type and ``mass-critical'' dispersive equations, where cocompactnes of Strichartz imbeddings was proved and employed by Terence Tao and
his collaborators.

The key element of the cocompactness theory is the premise that compactness of imbeddings of two functional spaces, 
which in many cases is attributed to the scaling invariance $u\mapsto t^{r}u(t\cdot)$ (that leads to localized non-compact 
sequences of ``blowups'' or ``bubbles'' $t_k^{r}w(t_k\cdot)$, $t_k\to\infty$), 
can be caused by invariance with respect to any other group of operators
acting isometrically on two imbedded spaces $X\hookrightarrow Y$. 
Furthermore, proponents of the cocompactness theory insist that there are many concrete applications which involve
such operators (``gauges'' or ``dislocations'') that are quite different from the Euclidean blowups. Indeed, 
recent literature contains concentration analyis of sequences in Sobolev and Strichartz spaces,
involving  actions of anisotropic or inhomogeneous dilations, of isometries of Riemannian manifolds (or more generally, of conformal groups 
on sub-Riemannian manifolds and other metric structures) and of transformations in the Fourier domain.

Improvement of convergence, based on elimination of concentration (understood in the abstract sense as terms of the form
$g_kw$, where $\{g_k\}$ is a non-compact sequence of gauges) can be illustrated in the sequence spaces on an elementary example 
based on Proposition 1 of \cite{Jaffard}.
\begin{ex}
The imbedding $\ell^{p}(\Z)\hookrightarrow\ell^{q}(\Z)$,
$1\le p<q\le\infty$ is not compact, since sequences of the form $u(\cdot +k)$
converge to zero weakly in $\ell^p$, but have constant $\ell^\infty$-norm.
Let us decide that $u(\cdot+j_k)$ with $|j_k|\to\infty$ is a typical ``concentrating behavior`` and 
eliminate it from a given sequence  $u_{k}\in\ell^{p}$ 
by assuming that for any sequence
$j_{k}\in\Z$, one has $u_{k}(\cdot+j_{k})\rightharpoonup0$ in $\ell^{p}$. 
Then $u_{k}(j_{k})\to0$ in $\R$ for any sequence $j_{k}\in Z$, 
which implies $u_{k}\to0$ in $\ell^{\infty}$. Since
$\|u\|_{q}^{q}\le\|u\|_{\infty}^{q-p}\|u\|_{p}^{p}$, one also has
$u_{k}\to0$ in $\ell^{q}$ for any $q>p$. 
We conclude that elimination of possible ``concentration'' caused by shifts $u\mapsto u(\cdot+j),\; j\in\Z$,
assures convergence in $\ell^{q}$.
 \end{ex}
This example gives motivation to the following definitions.

\begin{defn}(Gauged weak convergence.) Let $X$ be a Banach
space, and let $D$ be a bounded set of bounded linear operators on
$X$ containing the identity operator. One says that a sequence $u_{k}\in D$
converges to zero $D$-weakly if $g_{k}u_{k}\rightharpoonup0$ with
any choice of sequence $(g_{k})\subset D$. We denote the gauged weak
convergence as $u_{k}\cw0$.
\end{defn}
\begin{defn}
Let $X$ be a Banach space continuously imbedded
into a Banach space $Y$. One says that the imbedding $X\hookrightarrow Y$
is cocompact (relative to the set $D$) if $u_{k}\cw0$ implies $\|u_{k}\|_{Y}\to0$.
\end{defn}
\begin{defn}
One says that the norm of $Y$ provides a (local) metrization
of the $D$-weak convergence on $X$, if for any bounded sequence $u_{k}\in X$,
\[
u_{k}\cw0\Longleftrightarrow\|u_{k}\|_{Y}\to0.
\] 
\end{defn}
Note that we speak only about \emph{local} metrization, on the balls of $X$, and
that convergencies in different norms, for sequences restricted to such balls, become equivalent. For example,
all $\ell^{q}$-convergences with $q>p\ge1$ are equivalent on a ball of $\ell^{p}$, 
and all $L^{q}$-convergences with $q\in(p,\frac{pN}{N-p})$
are equivalent on a ball of $W^{1,p}(\R^{N})$, $1\le p<N$ (by
the limiting Sobolev imbedding and the H\"older inequality).
\begin{ex}
Let $X=H_{0,\mathrm{rad}}^{1}(B)$ be the subspace of all
radial functions in the Sobolev space $H_{0}^{1}(B)$, let $Y=L^{6}(B)$, where $B\subset\R^{3}$
is a unit ball, and let $D=\lbrace h_{t}u(x)=t^{1/2}u(tx)\rbrace_{t>0}$. \\
The lack of compactness of the imbedding $X\hookrightarrow Y$ is demonstrated by the blowup
sequences $t_{k}^{1/2}w(t_{k}x)$ with $t_{k}\to\infty$ and $w\in\mathcal{D}_{rad}^{1,2}(\mathbb{R}^{3})\setminus\{0\}$.
If, however, $u_k$ is a bounded sequence in $X$ and all its ``deflation sequences`` $t_{k}^{-1/2}u_{k}(t_{k}^{-1}x)$,
with $t_{k}\to\infty$, converge  to zero weakly in $X$, the sequence $u_k$ has a subsequence convergent in $Y$.
This we express as $D$-cocompactness of the imbedding $H_{0,\mathrm{rad}}^{1}(B)\hookrightarrow L^{6}(B)$. 
See Proposition~\ref{prop:radial} below with an elementary proof.
\end{ex}
While the term \emph{cocompact imbedding} is recent, the property
itself has been known for Sobolev spaces for decades, and can traced to a lemma 
by Lieb \cite{Lieb}. Cocompactness of imbeddings relative to rescalings (actions of translations
and dilations), which for the Sobolev spaces is usually credited to Lions \cite{PLL2a}, is known today 
for a range of Besov and Triebel-Lizorkin
spaces as $X$ and different Besov, Triebel-Lizorkin, $ $Lorentz
and BMO spaces as $Y$ (\cite{BCK}, once we note that the crucial Assumption 1 in this paper, 
expressed in terms of the wavelet bases implies cocompactness and is in a general case equivalent to it,
see the argument in subsection 2.2 below). 
Cocompactness of Sobolev imbeddings is known also for function spaces on manifolds, with
the role of translations and dilations taken over by the conformal group. Cocompactness is also established in 
Trudinger-Moser imbeddings (with inhomogeneous dilations) and in Strichartz imbeddings in dispersive equations. 
Section 2 gives a summary of known cocompact imbeddings. 
\par
Section 3 studies profile decompositions that arise in the presence of cocompact imbeddings. These are largely functional-analytic results,
established in two general cases, for Hilbert spaces and for Banach spaces imbedded into $L^p$-spaces. 
Beyond that, profile decompositions have been proved for a wide range of imbeddings for spaces that have a wavelet basis of rescalings,
presented in \cite{BCK} (with a remainder that vanishes in a weaker sense than in the $Y$-norm),  
as well as for some Trudinger-Moser and Strichartz imbeddings. First profile decompositions in literature were found
for specific sequences, typically, Palais-Smale sequences for semilinear elliptic functionals (Struwe \cite{Struwe84}, Brezis
and Coron \cite{BrezisCoron}, Lions himself \cite{Lions87} and Benci \& Cerami
\cite{BenciCerami}). The first profile decomposition for {\em general} bounded sequences in 
$\mathcal{D}^{1,p}(\mathbb{R}^{N})$  was proved by  Solimini \cite{Solimini}, and, furthermore, from Solimini's argument 
it  also became clear that a profile decomposition is essentially a functional-analytic phenomenon, 
and it is the cocompactness (still not a named property) that requires a substantial hard analysis. 
\par
Section 4 deals with reduction of cocompactness and profile decomposition to subspaces, which in many cases results 
in more specific types of concentration or even if disappearance, as, for example, in case of the radial subspaces of Sobolev spaces 
(Strauss Lemma), as well as list few sample arguments that allow to derive cocompactness by transitivity of imbeddings or interpolation. 
\par
The range of applications of concentration analysis is wider than the scope of this survey. We do not consider here in any great detail 
time evolution of concentration in the initial data that arises in semilinear dispersive equations, 
large-time emergence of concentration in evolution equations, applications to geometric problems, and blow-up arguments for sequences of 
solutions of PDE that may benefit from further built-in structure of the equations. 

\section{Cocompact imbeddings}

\subsection{Example: cocompactness of Sobolev imbeddings with elementary poofs}
The following theorem is essentially a lemma of Lieb from \cite{Lieb}. 
\begin{thm}
 \label{thm:shifta}
Let 
\begin{equation}
D=\lbrace u\mapsto u(\cdot-y),\; y\in\mathbb{Z}^{N}\rbrace.\label{eq:shifts}
\end{equation}
Assume that $N>p>1$. The Sobolev imbedding $W^{1,p}(\mathbb{R}^{N})\hookrightarrow L^{q}$,
$p\le q\le p^{*}=\frac{pN}{N-p}$, is cocompact relative to the group
$D$, unless $q=p$ or $q=p^{*}$. The $L^{q}$-norm gives a $W^{1,p}$-local
metrization of the $D$-weak convergence.
\end{thm}
The original statement asserted only convergence in measure, but under a $W^{1,p}$-bound, so that convergence in measure implies
convergence in $L^q$, $p\le q\le p^{*}=\frac{pN}{N-p}$. Moreover, the original statement (and many other
cocompactness statements in literature) is expressed in terms of negation: if a bounded sequence in 
$W^{1,p}$ does not converge in measure, then it has a subsequence with a nonzero, under suitable translations, weak limit. 
\begin{proof} Let $Q=(0,1)^{N}$. Assume that and $u_{n}(\cdot-y_{n})\rightharpoonup0$
in $W^{1,p}(\mathbb{R}^{N})$ for any sequence $y_{n}\in\mathbb{Z}^{N}$.
For every $y\in\mathbb{Z}^{N}$we have by the Sobolev inequality 
\[
\int_{Q+y}|u_{n}|^{q}\le\int_{Q+y}(|\nabla u_{n}|^{p}+|u_{n}|^{p})\;\left(\int_{Q+y}|u_{n}|^{q}\right)^{1-p/q}.
\]
Adding up the inequalities over $y\in\mathbb{Z}^{N}$ while estimating
the last factor by the supremum over $y$, and replacing the supremum
by twice the value of the term at the ``worst'' values of $y_{n}$,
we get

\[
\int_{\R^{N}}|u_{n}|^{q}\le\left(\int_{Q}|u_{n}(\cdot-y_{n})|^{q}\right)^{1-p/q},
\]
which converges to zero as a consequence of compactness
of the subcritical Sobolev imbedding over $Q$.
\end{proof}
This is a model proof of cocompactness that uses partitioning of
the norm of the target space into ``cells of compactness'', in this
case involving the fundamental domain of the lattice group acting
on $\mathbb{R}^{N}$, where one can benefit from compactness, followed
by a reassembly of the full norm from the vanishing terms. In the wavelet approach presented in subsection 2.2
the role similar to that of ``compactness cell'' is played by the mother wavelet.
\par
The limiting Sobolev imbedding $\mathcal D^{1,p}(\R^N)\hookrightarrow L^\frac{pN}{N-p}(\R^N)$ is cocompact
only if one enlarges the group $D$ by the action of dilations. The following statement originates in Lions \cite{PLL2a} with 
three different proofs given later by Solimini \cite{Solimini}; G\'erard \cite{Gerard} (for $p=2$) and Jaffard \cite{Jaffard}; 
and the author \cite{ccbook}. 
\begin{thm}
Let $\gamma>0$, $N>p\ge1$,
$r=\frac{N-p}{p}$, and let 
\begin{equation}
D=\lbrace u\mapsto\gamma^{rj}u(\gamma^{j}(\cdot-y)),\; y\in\mathbb{R}^{N},j\in\mathbb{Z}\rbrace.\label{eq:rescalings}
\end{equation}
The imbedding $\mathcal{D}^{1,p}(\mathbb{R}^{N})\hookrightarrow L^{\frac{pN}{N-p}}$
is cocompact relative to the group $D$. Furthermore, the $L^\frac{pN}{N-p}$-norm
on bounded subsets of $\mathcal{D}^{s,p}(\mathbb{R}^{N})$ provides
a metrization of $D$-weak convergence.
\end{thm}
This result is usually rendered in literature with the scaling factor in $(0,\infty)$ instead of the discrete
subset $\gamma^{\Z}$. Sufficiency of the discrete values can be observed by following the available proofs, or derived a posteriori.
We give below a proof the statement reduced to the case of the radial subspace of $\mathcal{D}^{1,p}$.
\begin{prop}
\label{prop:radial}
Let $\gamma>1$, $N>p>1$, and let 
\[
D=\lbrace u\mapsto\gamma^{j}u(\gamma^{\frac{p}{N-p}j}\;\cdot),\;j\in\mathbb{Z}\rbrace.
\]
The limiting Sobolev imbedding of the radial subspace of $\mathcal{D}^{1,p}(\mathbb{R}^{N})$, 
$\mathcal{D}_{\mathrm{rad}}^{1,p}(\mathbb{R}^{N})\hookrightarrow L^{\frac{pN}{N-p}}(\mathbb{R}^{N})$,
is cocompact relative to the group $D$ and the $L^{\frac{pN}{N-p}}$-norm
provides a $\mathcal{D}^{1,p}(\mathbb{R}^{N})$-local metrization of
the $D$-weak convergence. 
\end{prop}
\begin{proof}Let $u_{k}\in\mathcal{D}_{\mathrm{rad}}^{1,p}(\R^{N})$ and
assume that with any $j_{k}\in\Z$, $\quad\gamma^{\frac{N-p}{p}j_{k}}u_{k}(\gamma^{j_{k}}\cdot)\rightharpoonup0.$
Let  $p^{*}=\frac{pN}{N-p}$ and apply the limiting Sobolev inequality to $\chi_{j}(u)=\gamma^{j}\chi(\gamma^{-j}|u|)$,
where $\chi\in C_{0}^{\infty}((\frac{1}{\gamma},\gamma^{2}))$ satisfying $\chi(s)=s$
for $s\in[1,\gamma]$:
\[
\left(\int_{|u_{k}|\in(\gamma^{j},\gamma^{j+1})}|u_{k}|^{p^{*}}\mathrm{d}x\right)^{\frac{p}{p^{*}}}\le 
C\int_{|u_{k}|\in(\gamma^{j-1},\gamma^{j+2})}\left(|\nabla u_{k}|^{p}+\frac{|u_{k}|^{p}}{|x|^{p}}\right)\mathrm{d}x.
\]
Taking the sum over $j\in\Z$ and estimating the last integrand in the right hand side
by the Hardy inequality, we have
\[
\int_{\R^{N}}|u_{k}|^{p^{*}}\le C\int_{\R^{N}}|\nabla u_{k}|^{p}
\left(\sup_{j\in\Z}
\int_{|\gamma^{j}u_{k}(\gamma^{\frac{p}{N-p}j}x)|\in(1,\gamma)}|\gamma^{j}u_{k}(\gamma^{\frac{p}{N-p}j}x)|^{p^{*}}\mathrm{d}x
\right)^{1-\frac{p}{p^{*}}}.
\]
Replace the supremum in the right hand side by twice the value at a suitable sequence $j_k$.
The integrand in this term is a radial function, uniformly bounded in $L^1\cap L^\infty$, which, 
as a rescaled weakly vanishing sequence, converges to zero pointwise. Consequently, the left hand side converges to zero and 
cocompactness is proved.

To verify the local metrization, assume that $u_{k}\to0$ in $L^{p^{*}}(\R^{N})$ and that
the sequence is bounded in $\mathcal{D}^{1,p}$. Then, for any sequence
$g_{k}\in D$, the scaling invariance of the $L^{p^{*}}$-norm implies that
$g_{k}u_{k}\rightharpoonup0$ in $L^{p^{*}}$. Then, by density, $g_{k}u_{k}\rightharpoonup0$ in $\mathcal{D}^{1,p}$,
and thus $u_{k}\cw0$. 
\end{proof}
\subsection{Cocompactness relative to rescalings: the wavelet approach}
Imbeddings of Sobolev type are cocompact relative to the rescalings
group \eqref{eq:rescalings} for a large class of spaces. The first
results of this type were obtained for imbeddings of Riesz potential spaces $\dot{H}^{s,p}(\R^{N})$
into $L^{\frac{pN}{N-sp}}(\R^{N})$, $s>0$, $1<p<N$, 
G\'erard \cite{Gerard} (for Sobolev spaces with $p=2$) and Jaffard \cite{Jaffard} (for general $p$).
Here we summarize a major generalization of their analysis in Bahouri, Cohen and
Koch \cite{BCK}. 

Let $X\hookrightarrow Y$ be two Banach spaces of functions on $\R^{N}$, and assume that there 
exists an unconditional Schauder wavelet basis of wavelets $\Gamma\psi$, same for $X$ and $Y$, 
where
\begin{equation}
\Gamma=\lbrace u\mapsto g_{j,y}u=2^{rj}u(2^{j}\cdot-y)\rbrace_{j\in\Z,y\in\Z^{N}}\label{eq:BCK-D},
\end{equation}
and $r>0$ and $\psi\in X$ (``mother wavelet'') are fixed.
It's assumed that operators $u\mapsto u(\cdot-y)$, $y\in\R^{N}$,
and $u\mapsto t^{r}u(t\cdot)$, $t>0$, are
isometries in both $X$ and $Y$. For each $M\in\N$ and for every function $u\in X$, 
expanded in the basis $\Gamma\psi$ as $u=\sum_{g\in \Gamma}c(g)g\psi$, define a subset $\Gamma_{M}(u)\subset\Gamma$
of cardinality M which corresponds to the M largest values of $|c(g)|$,
and set 
\[
Q_{M}u=\sum_{g\in\Gamma_{M}(u)}c(g)g\psi.
\]
Note that such set $\Gamma_{M}$ always exists (and is not unique
when some $|c(g)|$ are equal, so one fixes it arbitrarily) due to
the fact that $\Gamma\psi$ is a Schauder basis for $X$, and thus 
for any $\eta>0$ only finitely many coefficients $c(g)$
have their absolute value larger than $\eta$. 

The main condition in \cite{BCK}, required for the imbedding $X\hookrightarrow Y$ (Assumption 1), is
\begin{equation}
\sup_{\|u\|_{X}\le1}\|Q_{M}u-u\|_{Y}\to0\text{ as }M\to\infty.\label{eq:Assumption1}
\end{equation}
\begin{rem}
When $X$ is reflexive, condition \eqref{eq:Assumption1} implies that
the imbedding $X\hookrightarrow Y$ is $D$-cocompact,  where $D$ is the rescaling group \eqref{eq:rescalings} with $\gamma=2$.
Indeed, assume that $h_{k}u_{k}\rightharpoonup$0 in $X$ for every rescaling sequence
$\{h_{k}\}\subset D$. Let $\Gamma'=\lbrace g^{-1}:\; g\in\Gamma\}\subset D$,
and note that $\Gamma'\Gamma\psi\subset\Gamma\psi$. 
Consider the wavelet expansion coefficients $c_{k}(g)$, $g\in\Gamma$,
of functions $u_{k}$, and let $g_{k}\in\Gamma$ be such that 
$|c_{k}(g_{k})|=\eta_{k}:=\max_{g\in\Gamma}|c_{k}(g)|=\max_{g\in\Gamma_{M}}|c_{k}(g)|$.
Then the coefficient $c'_{k}(\mathrm{id})$, corresponding to the
basis vector $\psi$ in the wavelet expansion for $g_{k}^{-1}u_{k}$,
is equal to $c(g_{k})$. At the same time $c'_{k}(\mathrm{id})\to0$
since since the coefficients of expansions
in a Schauder basis of a reflexive space are continuous
linear functionals and $g_{k}^{-1}u_{k}\rightharpoonup0$. 
Fix now $\epsilon>0$ and let $M=M(\epsilon)$ be such that $\|Q_{M}u_{k}-u_{k}\|_{Y}\le\epsilon$.
Then
\[
\|u_{k}\|_{Y}=\|Q_{M}u_{k}+(u_{k}-Q_{M}u_{k})\|_{Y}\le\|Q_{M}u_{k}\|_{Y}+\epsilon\le M(\epsilon)\eta_{k}+\epsilon.
\]
Let $k\to\infty$ and note that $\epsilon$ was arbitrary.
\par
Conversely, cocompactness of the imbedding relative to rescalings
easily implies \eqref{eq:Assumption1} if 
the space $X$ and its basis $\Gamma\psi$ satisfy the following condition: 
$Q_{M}u_{M}-u_{M} \rightharpoonup 0$ as $M\to\infty$ for any bounded sequence $\lbrace u_{M}\rbrace\subset X$.
Details are left to the reader.  
\end{rem}
For spaces brought up below, existence of an unconditional 
Schauder basis of rescaling wavelets is known, and the norms in their spaces have equivalent definitions in terms
of expansion coefficients in the wavelet basis (see \cite{Meyer24}). Using these characterizations, \cite{BCK} verifies condition
\eqref{eq:Assumption1}, and thus cocompactness of the imbeddings whenever $X$ is reflexive.  
Results of \cite{BCK} contain several earlier wavelet-based cocompactness results, in particular, \cite{Gerard, Jaffard, dVJP-9, Kyriasis19}. 
Some cases follow from the others via a simple transitivity argument: if $X\hookrightarrow Y$, $Y\hookrightarrow Z$ and one of the imbeddings
satisfies \ref{eq:Assumption1} (or is cocompact), then satisfies \ref{eq:Assumption1} (resp. is cocompact), 
or from the equivalence of different $Y$-convergences on bounded sets of $X$.
\begin{thm}
\label{thm:bck}
Let $p,q\in[1,\infty]$, $a,b\in(0,\infty]$, $s>t\ge0$, and $r=\frac{N}{p}-s>0$. 
The following imbeddings satisfy \eqref{eq:Assumption1}
(and are cocompact relative to rescalings whenever the domain is a reflexive space).\\
(i) \textbf{$\dot{B}{}_{p,p}^{s}(\R^N)\hookrightarrow L^{q}(\R^N)$}, \textbf{$q<\infty$,
$\frac{1}{p}-\frac{1}{q}=\frac{s}{N}>0$}.\\
(ii) \textbf{$\dot{B}{}_{p,p}^{s}(\R^N)\hookrightarrow\dot{B}_{q,q}^{t}(\R^N)$},\textbf{
$\frac{1}{p}-\frac{1}{q}=\frac{s-t}{N}>0$}.\\
(iii) \textbf{$\dot{B}{}_{p,q}^{s}(\R^N)\hookrightarrow\dot{F}_{q,b}^{t}(\R^N)$},\textbf{
$\frac{1}{p}-\frac{1}{q}=\frac{s-t}{N}>0$}.\\
(iv) \textbf{$\dot{B}{}_{p,a}^{s}(\R^N)\hookrightarrow\dot{B}_{q,b}^{t}(\R^N)$},\textbf{
$\frac{1}{p}-\frac{1}{q}=\frac{s-t}{N}>0$}, $a<b$ (no cocompactness
when $a=b$).\\
(v) $\dot{B}_{p,p}^{s}(\R^N)\hookrightarrow\mathrm{BMO}(\R^N)$, $s=\frac{N}{p}>0$.\\
(vi) \textbf{$\dot{B}{}_{p,a}^{s}(\R^N)\hookrightarrow L^{q,b}(\R^N)$},\textbf{ $\frac{1}{p}-\frac{1}{q}=\frac{s}{N}>0$},
$a<b$.\\
(vii) \textbf{$\dot{F}{}_{p,a}^{s}(\R^N)\hookrightarrow\dot{F}_{q,b}^{t}(\R^N)$},\textbf{
$\frac{1}{p}-\frac{1}{q}=\frac{s-t}{N}>0$}, $a,b>0$.
\end{thm}
Note that for $m\in\N$, the space $\dot{F}_{p,2}^{m}$ coincides
with the Sobolev space $\mathcal{D}^{m,p}(\mathbb{R}^{N})$, and $\dot{F}_{q,2}^{0}$
coincides with $L^{q}$, so the last imbedding includes the limiting
Sobolev imbeddings. 
Profile decompositions for these imbeddings, which we discuss in Section 3, are also given in \cite{BCK}.

\subsection{Cocompactness of trace imbeddings}
\begin{thm}
 Let $1<p<N$ and $\bar{p}=\frac{p(N-1)}{N-p}$. The following imbeddings are cocompact:\\
(i) $W^{1,p}(\R^{N})\hookrightarrow L^{q}(\R^{N-1})$ , $q\in(p,\bar{p})$,
   relative to the lattice shifts $\lbrace u\mapsto u(\cdot-y)\rbrace_{y\in\R^{N-1}}$.\\
  \item (ii) $\mathcal{D}^{1,p}(\R^{N})\hookrightarrow L^{\bar{p}}(\R^{N-1})$
    relative to the rescalings group \eqref{eq:rescalings}. 
\end{thm}
\begin{proof}
The proof of (i) is repetitive of the argument in the paragraph 2.3.1 above,
employing the cubic lattice neighborhood of the hyperplane, instead of the cubic tessellation of the whole $\R^{N}$. 
The proof of (ii) for $p=2$ is given in Lemma~5.10 in \cite{ccbook}, and extends trivially
to general $p>1$. 
\end{proof}
Results in both cases easily extend to the case of hyperplanes of any dimension $d>p$. In both cases the corresponding 
$L^{p}$-norms give metrization of $D$-weak convergence. We are not aware of further results in literature
on cocompactness of trace imbeddings, but it is plausible that most trace imbeddings for Besov
and Triebel-Lizorkin spaces are similarly cocompact, and the wavelet argument of \cite{BCK} can be applied here as well.

\subsection{Cocompactness of Sobolev imbeddings on metric structures}
For the subcritical Sobolev imbeddings on manifolds, the actions of isometries
play the same role in cocompactness as parallel translations in the Euclidean case. 
Let $M$ be a complete Riemannian manifold and let $I(M)$ be its isometry group.
The following result deals with cocompactness of ``magnetic'' Sobolev spaces, corresponding in appropriate cases to the Schr\"odinger
operator with external magnetic field. 
Magnetic Sobolev space $W_{\alpha}^{1,p}(M)$, with a fixed $\alpha\in T^{*}M$, called magnetic potential,
is the space of functions $M\to\mathbb{C}$ with measurable weak derivatives,
characterized by the finite norm
\[
\|u\|_{1,p}^{p}=\int_{M}(|du+iu\alpha|^{p}+|u|^{p}).
\]
The usual Sobolev space corresponds to the form $\alpha$ being exact (i.e. zero modulo gauge transformation).
By the diamagnetic inequality $|du+i\alpha u|\ge|d|u||$, so one always has
$W_{\alpha}^{1,p}(M)\hookrightarrow W^{1,p}(M)$. 
When $\alpha$ is not exact, assume in addition that $M$ is simply
connected. Then for each diffeomorphism $\eta:\, M\to M$,
such that $d(\alpha\circ\eta)=d\alpha$, there exists a unique real-valued
function $\varphi_{\eta}\in C^{\infty}(M)$ satisfying $d\varphi=\alpha\circ\eta-\alpha$.
An elementary computation shows that if $\eta\in I(M)$, then
$W_{\alpha}^{1,p}(M)$-norm is preserved by the operators $u\mapsto e^{i\varphi_{\eta}}u\circ\eta$, 
called {\em magnetic shifts} (see \cite{ArioliSz}). In physics, the meaning of the relation
$d\alpha\circ\eta=d\alpha$ is that the magnetic field $d\alpha$ is periodic under isometries of
$M$. 
\begin{thm}
 Let $M$ be a complete smooth non-compact Riemannian $N$-manifold, cocompact (periodic) relative to $I(M)$, and 
 simply connected whenever $\alpha\in T^*M$ is not exact.  
Let $G$ be a closed subgroup of $I(M)$ and let $Q\subset M$ be a bounded set such that $GQ=M$. Define
\begin{equation}
 \label{DG}
D_G=\lbrace u\mapsto e^{i\varphi_{\eta}}u\circ\eta,\;\eta\in G\rbrace.
\end{equation}
Then the imbedding $W_{\alpha}^{1,p}(M)\hookrightarrow L^{q}(M)$,
$1<p\le q\le p^{*}=\frac{pN}{N-p}$, $N>p$, is cocompact relative to the set
of operators $D_{G}$. Moreover, the $L^{q}$-norm provides metrization of the $D$-weak convergence
on bounded subsets of $W_\alpha^{1,p}$. 
\end{thm}
\begin{proof}
The proof of cocompactness and verification of metrization for $p=2$ is given in \cite{ccbook}, Lemma~9.4. The proof for general $p$ is completely analogous. 
The argument is similar to that for Theorem~\ref{thm:shifta}, and involves existence of a covering for $M$ of 
uniformly finite multiplicity by sets $\lbrace\eta V\rbrace_{\eta\in G'}$, with some set $G'\subset G$ and some open set $V\subset M$.
We give here an argument for reduction of the general (magnetic) case to the case $\alpha=0$.
\par
If $e^{i\varphi_{\eta_{k}}}u_{k}\circ\eta_{k}\rightharpoonup0$
in $W_{\alpha}^{1,p}(M)$, then $|u_{k}\circ\eta_{k}|$ converges
to zero a.e. and, by the diamagnetic inequality, is bounded in $W^{1,p}(M)$.
Then, from cocompactness of the ``non-magnetic'' imbedding $W^{1,p}(M)\hookrightarrow L^{q}(M)$, follows
$|u_{k}|\to0$ in $L^{q}(M)$, $q\in(p,p^{*})$.  
\end{proof}
Cocompactness of subcritical imbeddings relative to isometries extends to Sobolev spaces on metric structures 
other than Riemannian manifolds, in particular, to the Sobolev spaces of the Kohn
Laplacian  on Carnot groups (\cite{LieGroups}) and of blowups of a self-similar
fractal of a class involving Sierpinski gasket, as in \cite{Frac};
see also the paper \cite{BiSchiTi} in the setting of axiomatic Sobolev
spaces on general metric structures.
\par
Let now $M$ be a simply connected nilpotent stratified Lie group (Carnot group) with
a stratification $T_{e}M=\bigoplus_{j=1}^{m}Y_{j}$ . Let $\nu=\sum_{j=1}^{m}j\dim Y_{j}\ge3$.
One calls the diffeomorphism $T_{s}=\exp_{e}\tau_{s}\exp_{e}^{-1}$
of $M$ an \emph{anisotropic dilation} if $\tau_{s}$ is given by
$\tau_{s}|_{Y_{j}}=s^{j}$, $j=1,\dots,m$. Let $v_{1},\dots v_{\dim Y_{1}}\in T_{e}M$
be an orthonormal basis in $Y_{1}$ and consider the subelliptic Sobolev
space $\dot{H}^{1}(M)$, characterized by the norm 
\[
\|u\|^{2}=\int_{M}\left|\sum_{i=1}^{\dim Y_{1}}|\langle du,v_{i}\rangle|^{2}\right|.
\]
In particular, for the Heisenberg group $\mathbb{H}_{N}$, identified
as $\mathbb{R}^{N}\times\R^{N}\times\R$ with the group law 
\[
(x,y,t)(x',y',t')=(x+x',y+y',t+t'+2x\cdot y'-2x'\cdot y),
\]
the first stratum is $Y_{1}=\mathbb{R}^{N}\times\R^{N}$, the effective
dimension $\nu$ equals $2N+2$, the anisotropic dilations are
$T_{s}(x,y,t)=(sx,sy,s^{2}t)$, and the Sobolev norm is given by
\[
\|u\|^{2}=\int_{\R^{2N+1}}\left(\sum_{i=1}^{N}|(\partial_{x_{i}}+2y_{i}\partial_{t})u|^{2}+\sum_{i=1}^{N}|(\partial_{y_{i}}-2x_{i}\partial_{t})u|^{2}\right)dxdydt.
\]
The space $\dot H^{1}(M)\hookrightarrow L^{\frac{2\nu}{\nu-2}}(M)$ is
continuously imbedded into $H^{1}(M)$. For further details on
subelliptic Sobolev spaces we refer to (\cite{Folland}, \cite{FollandStein}, as well a brief exposition in \cite{ccbook}, Chapter 9.
\begin{thm}
Let $\dot H^{1}(M)$ be the subelliptic Sobolev space on the  Carnot group $M$ as above, let $\gamma>1$
and let 
\[
D=\lbrace u\mapsto\gamma^{\frac{\nu-2}{2}j}u\circ\eta\circ T_{\gamma^{j}},\;\eta\in M,\, j\in\Z\rbrace.
\]
Then the imbedding ${\dot{H^{1}}(M)\hookrightarrow L^{\frac{2\nu}{\nu-2}}(M)}$
is cocompact relative to $D$. Furthermore, the $L^{\frac{2\nu}{\nu-2}}-$norm
provides a metrization of $D$-weak convergence on bounded subsets
of $\mathcal{D}^{\ell,p}(\mathbb{R}^{N})$. 
\end{thm}
This result is proved in \cite{ccbook}, Lemma~9.14, for $p=2$ (extension to $p\neq 2$ is elementary). 
Earlier results can be found in \cite{LieGroups} for the subcritical case and the general Carnot group, 
and in \cite{BenAmeur} for the case of Heisenberg group). Like in the Euclidean case, many applications
to subelliptic PDE on Lie groups could be handled with the help of the Willem-Chabrowski version of concentration compactness (e.g.
in papers of Garofalo et al) or with Struwe's ``global compactness`` (\cite{Struwe84}, see Theorem~3.1 in the book \cite{DruHeRo} of Hebey, Druet and
Robert), which is a realization of a possible
more general profile decomposition for manifolds, for the case of Palais-Smale sequences with dilations expressed via the exponential map.
\subsection{Cocompactness of the Moser-Trudinger imbedding}
The counterpart of Sobolev imbedding of $\mathcal{D}^{1,p}(\R^{N})$
in the borderline case  $p=N$, is the imbedding defined by the Moser-Trudinger inequality (Yudovich, \cite{Yudovich}, 
independently rediscovered by Pohozhaev, Peetre and Trudinger, and with the optimal exponent proved by Moser \cite{Moser}).
Moser-Trudinger inequality is stated for bounded domains and it is false for $\R^N$. There reason for that is that the gradient norm $\|\nabla u\|_N$
on $C_0(\R^ N)$ does not dominate any linear functional $\langle\varphi,u\rangle$ with  $\varphi\in\mathcal D'(\R^N)\setminus\{0\}$, 
i.e. the space $\mathcal D^{1,N}(\R^N)$ defined as a formal completion of $C_0^\infty(\R^N)$ in the gradient norm, 
is not continuously imbedded even into the space of distributions. 
Another way to express this is that there exists a Cauchy sequence representing zero of the completion, which converges to 
$1$ uniformly on every compact set. A counterpart of the Moser-Trudinger inequality for $\R^N$, proved by Li and Ruf \cite{Ruf}, 
involves the full Sobolev norm rather than the gradient norm, and it also avoids lower powers of $u$ 
which have poor integrability at infinity. In what follows we will use the Sobolev norm $\|u\|_{1,N}$ equal to the
standard Sobolev norm if the domain $\Omega\subset\R^{N}$ is unbounded and to the 
equivalent norm of $W_0^{1,N}(\Omega)$, $\|\nabla u\|_{N}$, if $\Omega$ is bounded. 
The inequality that expresses both the Moser-Trudinger inequality and the Li-Ruf inequality, is
\[
\sup_{u\in W_{0}^{1,N}(\Omega),\|u\|_{1,N}\le1}\int_{\Omega}\exp_{N}(\alpha_{N}|u|^{N'})\mathrm{d}x<\infty,
\]
where $\alpha_N=N\omega_{N-1}^{1/(n-1)}$, $\omega_{N-1}$ is the measure of the unit sphere in $\R^N$ , $N'=\frac{N}{N-1}$ and
$\exp_{N}(t)=e^{t}-\sum_{j=0}^{N-2}\frac{t^{j}}{j!}$. For bounded
domains one may equivalently use $e^{t}$ instead of $\exp_{N}(t)$, since the subtracted polynomial
also has a bounded integral.

The imbedding expressed by the Moser-Trudinger inequality is $W^{1,N}_0(\Omega)\hookrightarrow \exp_N L^{N'}$, where 
$\exp_{N}L^{N'}$ is the Orlicz space space associated with the convex function $\exp_{N}(t^{N'})$.

In the spaces of radial functions cocompactness of imbeddings is established
in the cases when $\Omega$ is a disk (without loss of generality
we consider here the unit disk $B$) or $\R^{N}$ (if $\Omega$ is
an annulus or an exterior disk, the imbedding is compact for elementary
reasons). The result below is due to \cite{ATdO}:
\begin{thm}
The imbedding $W_{0,\mathrm{rad}}^{1,N}(B)\hookrightarrow\exp L^{N'}(B)$
is cocompact relative to the group
\[
D=\lbrace u\mapsto s^{1-1/N}u(r^{s})\rbrace_{s>0}.
\]
Furthermore, a local metrization of the $D$- weak convergence is provided
by the norm $\sup_{0<r<1}\frac{|u(r)|}{(\log\frac{1}{r})^{1/N'}}$.
\end{thm}
One can easily derive from here cocompactness in $W_{\mathrm{rad}}^{1,N}(\R^{N})$ by considering a sequence $u_k-u_k(1)$ on $B$, 
bounded in $H_0^1(B)$, and noting that the restriction of a radial sequence $u_k\rightharpoonup 0$ in $W^{1,N}(\R^{N})$
to the complement of $B$ has a uniform bound and converges to zero in $L^{NN'}$. 
\begin{cor}
\label{2.10}
The imbedding $ W_{\mathrm{rad}}^{1,N}(\R^{N})\hookrightarrow\exp_{N}L^{N'}(\R^{N})$
is cocompact relative to the group
\[
D=\lbrace g_{s}:\; g_{s}u(r)=u(1)+s^{-1/N'}(u(r^{s})-u(1)),\; r\le1;\; g_{s}u(r)=u(r),\; r>1\rbrace_{s>0},
\] 
and metrization of the $D$- weak
convergence is acheived by the norm
\[
\sup_{0<r<1}\frac{|u(r)-u(1)|}{(\log\frac{1}{r})^{1/N'}}+\|u\|_{p,\R^{N}\setminus B}, p\in(N,\infty).
\]
\end{cor}
Without the assumption of radiality, a cocompactness result is known presently only for $N=2$.  
\begin{thm} (\cite{AT})
The imbedding $H_{0}^{1}(B)\hookrightarrow\exp L^{2}(B)$
is cocompact relative to the set of translations combined with transformations 
\begin{equation}
\label{eq:zj}
u(z)\mapsto j^{-1/2}u(z^{j}, j\in\N,
\end{equation}
where $z$ expresses coordinates of $B$ as a complex variable. 
A (quasi-)metrization of the $D$-weak convergence is acheived by the quasinorm
$\sup_{0<r<1}\frac{u^{\star}(r)}{(\log\frac{1}{r})^{1/2}}$, where
$u^{\star}$ denotes the symmetric decreasing rearrangement of $u$. 
\end{thm}
Note that the quasinorm $\sup_{0<r<1}\frac{u^{\star}(r)}{(\log\frac{1}{r})^{1/2}}$ 
is marginally stronger than the Zygmund quasinorm
$\sup_{0<r<1}\frac{u^{\star}(r)}{(1+\log\frac{1}{r})^{1/2}}$, which is equivalent
to the $\exp L^{2}$- norm (see \cite{BennettRudnick}). 
\subsection{Cocompactness of Strichartz imbeddings}
Consider a Strichartz imbedding that estimates the
space-time $L^{q}$- norm of the solution of the evolutionary Schr\"odinger
equation by the $L^{2}$-norm of the initial data (for details see \cite{Cazenave, KV}).
The following result is due to Terence Tao, \cite{Tao}, and the version of the proof in \cite{KV} optimizes the group.
From the presentation of \cite{KV} we could easily infer 
that cocompactness remains valid if one restricts dilations to a discrete group. In the theorem below $\hat u$ denotes
the Fourier transform.
\begin{thm} 
\label{thm:Tao}
Let $N\ge3$. The imbedding defined by the inequality
\[
\|e^{it\Delta}u\|_{L^{q}(\R^{N+1})}\le C\|u\|_{L^{2}(\R^{N})},\; q=\frac{2N+2}{N},
\]
is cocompact in $L^{2}$ with respect to the product group of the
following transformations ($\gamma>1$ is a fixed number):
\[
u(x)\mapsto\gamma^{\frac{Nj}{2}}u(\gamma^{j}x),\; j\in\Z;
\]
\[
u(x)\mapsto u(x-y),\; y\in\R^{N};
\]
\[
\hat{u}(\xi)\mapsto\hat{u}(\xi-\xi'),\;\xi'\in\R^{N}.
\] 
\end{thm}
Translations in the Fourier domain allows to consider functions
with support on a cube in the Fourier domain, for which convergence
in $L^{2}$ implies convergence in $C^m(\R^N)$ for any $m\in\N$. The main technical point in the proof
is the ``reassembly'' of the inequality from the ``cells of compactness'' using methods of harmonic analysis. 

\section{Profile decompositions}
Given an imbedding of a Banach space $X$ into a Banach
space $Y$ and a bounded set $D$ of bounded linear bijections $X\to X$, a profile decomposition is a representation of a bounded
sequence $\{u_k\}\subset X$ in the form 
\begin{equation}
\label{pd}
u_k=\sum_{n\in\N}g_k^{(n)}w^{(n)}+r_k, \;r_k\stackrel{Y}{\to}0,  
\end{equation}
where the terms $g_k^{(n)}w^{(n)}$, $n\in\N$ (which may be called {\em elementary concentrations}), are actions of 
sequences of operators $\{g_k^{(n)}\}_k$ on elements $w^{(n)}\in X$, (which may be called {\em concentration profiles}), that
satisfy ${g_k^{(n)}}^{-1}u_k\rightharpoonup w^{(n)}$. Furthermore, the elementary concentrations are expected 
to be asymptotically decoupled in the sense ${g_k^{(n)}}^{-1}g_k^{(m)}\rightharpoonup 0$.
Convergence of the remainder in $X$ should not be generally expected, as can be illustrated on the following
example. Let $X=\ell^p(\Z)$ and let $u_k(j)=1/k^{1/p}$ for $j=1,\dots,k$ taking zero values for all other $j$.
If $D$ is the set of shifts $u\mapsto u(\cdot-j)$, $j\in\Z$, then $u_k$ has no (nonzero) profiles under any shift sequence, so $r_k=u_k$
and $\|r_k\|_p=1$. At the same time $\|r_k\|_q\to 0$ for any $q>p$.
Many profile decompositions found in literature (typically in papers using the wavelet argument, starting with \cite{Gerard}) 
are stated with a weaker remainder, namely, in the form 
\begin{equation}
\label{wr}
u_k=\sum_{n=1}^Mg_k^{(n)}w^{(n)}+r^{(M)}_k, \;\lim_{M\to\infty}\limsup_{k\to\infty}\|r^{(M)}_k\|_Y=0. 
\end{equation}
In many significant cases, however, such as Palais-Smale sequences for elliptic problems, one can establish a lower bound on some
norm of the concentration profiles, which assures that only finitely many profiles are non-zero, in which case the weak remainder 
$r_k^{(M)}$ with $M$ equal to the number of nonzero profiles, is the same as the strong remainder. 

As it was established in \cite{Revista}, profile decompositions hold whenever $X$ is a Hilbert space under a general
condition on the set $D$. A similar result for Banach space is also expected to be true, under some, rather weak but yet
unknown general conditions on $X$. As a temporary fix we give here a profile decomposition in the case when $X$ is continuously
imbedded into $L^q(M)$, where $M$ is a measure space.
\par
The abstract profile decomposition gives a remainder $r_k$ that converges to zero $D$-weakly. Cocompactness 
is defined as the property of $X,Y,D$ that $D$-weak convergence in $X$ implies convergence in the norm of $Y$.
It is verified (although not under that name and often expressed in different terms) and employed also in the proofs of profile
decompositions for specific functional spaces, such as profile decompositions in the cited papers of Solimini and Bahouri, Cohen and Koch,
(the latter verifies the property named Assumption 1 which, as we shown above, implies cocompactness). It is not clear yet,
when there are profile decompositions with the weak, but not with the strong, remainder, but
we expect that this may occur only under some special conditions, such as absence of some strong convexity.
The abstract profile decompositions were not yet considered for non-reflexive spaces. 

For the sake of simplicity we here consider the case when the set $D$ consists of isometries on $X$, as this is the case most often 
studied in applications. A profile decomposition with quasi-isometric operators in the Hilbert space
is given in \cite{ccbook}, Theorem~3.1.

Let us define the general class of the set of isometric operators that yields general profile decompositions. 
As usual, pointwise (or strong) operator convergence in $X$, $g_{k}\stackrel{s}{\to}g$, 
is convergence $g_{k}x\to gx$ in $X$ for any $x\in X$.
\begin{defn}
A set $D$ of surjective linear isometries on a Banach space $X$, 
closed with respect to the pointwise operator convergence,
is called a \emph{dislocation} (or a \emph{gauge}) 
set if any sequence $g_k^{-1}$ with $g_{k}\in D$,
that does not converge weakly to zero, has a pointwise convergent
subsequence.
\end{defn}
All operator sets in the cocompactness results of the previous section
are known to be dislocation sets (see \cite{ccbook} for Euclidean
rescalings, anisotropic rescalings on Carnot groups, actions of isometries
on locally compact manifolds, and magnetic shifts, see \cite{ATdO} and
\cite{AT} for inhomogeneous dilations in the Moser-Trudinger settings, 
and \cite{KV} for shifts in the Fourier variable.)

\subsection{Hilbert space}
In the general Hilbert space, the following profile decomposition
holds. \cite{ccbook}, Corollary~3.2.

\begin{thm}
\label{hilbcc}
Let $D$ be a set of dislocations on a Hilbert space $H$ and
let $u_{k}$ be a bounded sequence in $H$. There is a renumbered
subsequence of $u_{k}$ and sequences $\{g_{k}^{(n)}\}\subset D$, $w^{(n)}\in H$,
$n\in\N$, $g_{k}^{(1)}=\mathrm{id}$, such that

\[
g_{k}^{(n)^{-1}}u_{k}\rightharpoonup w^{(n)},
\]

\[
g_{k}^{(n)^{-1}}g_{k}^{(m)}\rightharpoonup0\text{ whenever }m\neq n,
\]

\[
\|u_{k}\|^{2}=\sum_{n}\|w^{(n)}\|^{2}+\|r_k\|^2+o(1),
\]
where
\[
r_k:=u_{k}-\sum_{n}g_{k}^{(n)}w^{(n)}\cw0,
\]
and the series in the last expression converges in $H$ uniformly
with respect to $k$. If, in addition, an imbedding $H\hookrightarrow Y$ is $D$-cocompact, the last expression
vanishes in $Y$.
 \end{thm}

Note that $g_{k}^{(n)^{-1}}g_{k}^{(m)}\rightharpoonup0$ if and only
if $(g_{k}^{(m)}v,g_{k}^{(n)}w)\to0$ for any $v,w\in H$, so this
relation can be called asymptotic orthogonality.
For the group of Euclidean shifts, the relation of asymptotic orthogonality means $|y_k^{(n)}-y_k^{(m)}|\to\infty$.
More generally, for actions of isometries of a Riemannian manifold $M$, asymptotic orthogonality means that for some point $x_0\in M$
(and then for any other point) the sequence ${\eta_k^{(n)}}^{-1}\eta_k^{(m)}x_0$ has no convergent subsequence.
For the rescalings $u\mapsto 2^{rj}u(2^j(\cdot-y))$, $y\in\R^N$, $j\in\Z$, asymptotic orthogonality means 
\begin{equation}
\label{orth}
|j_k^{(n)}-j_k^{(m)}|+|y_k^{(n)}-y_k^{(m)}|\to\infty. 
\end{equation}
Analogous orthogonality condition arises on Carnot groups, with rescalings $2^{jr(\nu)}T_{\gamma^j}u\circ\eta$ 
involving inhomogeneous dilations and  left group shifts (\cite{ccbook}, Section 9.9)

\subsection{Banach space case}
There is no generalization of Theorem~\ref{hilbcc} for the general Banach
space, but there is one (\cite{Cwiti2}) for functional spaces cocompactly
imbedded into $L^{p}(M,\mu)$ where $(M,\mu)$ is a measure space.
Profile decompositions, that are not particular cases of Theorem~\ref{hilbcc} or Theorem~\ref{thm:Banach} below, are summarized, 
to a great extent, in \cite{BCK}, which is dedicated to the case of Euclidean rescalings on functions of $\R^N$. 
\begin{thm}
\label{thm:Banach}
Let $M$ be a measure space and let $D$ be a set of dislocations on a reflexive Banach space
$X$ continuously imbedded into $L^{p}(M)$ with some $p>1$,
and assume that the operators in $D$ are isometric in $L^{p}$ and
that the imbedding $X\hookrightarrow L^{p}$ is $D$-cocompact. let
$u_{k}$ be a bounded sequence in $X$. There is a renumbered subsequence
of $u_{k}$ and sequences $(g_{k}^{(n)})\subset D$, $w^{(n)}\in X$,
$n\in\N$, $g_{k}^{(1)}=\mathrm{id}$, such that

\[
g_{k}^{(n)^{-1}}u_{k}\rightharpoonup w^{(n)},
\]

\[
g_{k}^{(n)^{-1}}g_{k}^{(m)}\rightharpoonup0\text{ whenever }m\neq n,
\]

\[
\sum_{n}\|w^{(n)}\|_{p}^{p}\le\liminf\|u_{k}\|_{p}^{p},
\]

\[
u_{k}-\sum_{n}g_{k}^{(n)}w^{(n)}\stackrel{L^{p}}{\to}0,
\]
and the series in the last expression converges in $L^{p}$ uniformly
with respect to $k$. If $X$ is dense in $L^{p}$, then $L^{p}$
provides a metrization of the $D$-weak convergence. 
\end{thm}
As a corollary of this, with $\mathcal{D}^{1,p}(\R^N)\hookrightarrow L^\frac{pN}{N-p}(\R^{N})$
cocompactly relative to rescalings $u\mapsto 2^{rj}u(2^j(\cdot-y))$, $y\in\R^N$, $j\in\Z$, $r=\frac{N-p}{p}$, 
one has the profile decomposition of Solimini \cite{Solimini}. 
The orthogonality condition in Solimini's profile decomposition is \eqref{orth}. 
Cocompactness is proved in \cite{Solimini} for the imbedding into 
$L^{\frac{pN}{N-p},q}(\R^{N})$, $p<q\le\infty$, but on the bounded subsets of $\mathcal{D}^{1,p}(\R^N)$ convergence in all these
quasinorms is equivalent to that in $L^\frac{pN}{N-p}(\R^{N})$. 
\begin{rem}
In  Section 9.9 of \cite{ccbook} Solimini's profile decomposition is generalized to Sobolev spaces $\mathcal D^{1,2}$ of Carnot groups.
The proof of cocompactness there can be trivially extended, with suitable parameter changes, to $\mathcal D^{1,p}$ with general $p$, 
and then Solimini's profile decomposition for Carnot groups with anisotropic rescalings 
follows for any $p\in(1,\infty)$ from from Theorem~\ref{thm:Banach}. 
\end{rem}

\subsection{Wavelet bases and profile decompositions}
Profile decompositions for a cocompact imbedding of functional spaces of $\R^N$, relative to rescalings,
are established in \cite{BCK} (following a number of earlier results surveyed there,
starting with G\'erard's paper \cite{Gerard}) in the weak remainder form \eqref{wr} and 
with the asymptotic decoupling condition \eqref{orth}. 
Additional information about the profile decomposition is given there in form of
stability estimates, namely $\ell^{p}$ -bounds on the sequence $\lbrace\|w^{(n)}\|_{X}\rbrace_{n\in\N}$
in the  cases of Besov and Triebel-Lizorkin spaces, which indicates a possibility of a stronger remainder.
There is a recent work with an expressed objective to replace the weak remainder \eqref{wr} in the profile decomposition
with a strong remainder, Palatucci \& Pisante \cite{PP}, which accomplished this task for the profile decomposition of \cite{Gerard}, 
that is, for imbedding of the Bessel potential spaces $\dot H^s(\R^N)$, $s>0$. 

Profile decomposition of \cite{BCK} follows from the following assumptions on the function spaces $X\hookrightarrow Y$
of $\R^{N}$.
\begin{enumerate}
\item The norms of $X$ and $Y$ are invariant with respect to shifts $ $and
to dilations $t^{r}u(t\cdot)$ with some $r\in\R$;
\item There exists a function $\psi\in X$ such that the set $D\psi$ is
an unconditional Schauder basis on both $X$ and $Y$ where $D$ is
the set \eqref{eq:BCK-D};
\item The imbedding $X\hookrightarrow Y$ satisfies condition \eqref{eq:Assumption1} (which,
as we above, implies cocompactness of the imbedding whenever $X$ is reflexive);
\item There is a $C>0$ such that for any sequence $u_{k}$ bounded in $X$,
with expansion $u_{k}=\sum_{g\in D}c_{k}(g)g\psi$, whose coefficients
$c_{k}(g)$ converge, for each $g\in D$, to respective finite limits
$c(g)$ as $k\to\infty$, the series $\sum_{g\in D}c(g)g\psi$ converges
in $X$ with $\|\sum_{g\in D}c(g)g\psi\|_{X}\le C\liminf\|u_{k}\|_{X}$.
\end{enumerate}
Conditions above (and thus the profile decomposition \eqref{wr}) are verified in \cite{BCK} for all imbeddings listed in
Theorem~\ref{thm:bck} above. 
\subsection{Concentration in time-evolution problems}
Profile decompositions for sequences of initial data (bounded in respective
Sobolev norms) of dispersive evolution equations, such as nonlinear Schr\"odinger or wave equation, give rise to profile decomposition
of finite energy solutions with these data. This type of concentration analysis is usually called energy-critical and involves the usual
group of rescalings (see \cite{KenigMerle,Keraani,Gallagher-NS,Tao-multiattractors}), 
and, for a survey, \cite{KV}). While for a linear equation construction a ``time-evolved'' profile decomposition 
in the energy space of initial data is straightforward, further technical effort is unvolved in extending such decomposition to
solutions of equations with a non-linear term of critical growth. Below we quote a representative result from \cite{BahouriGerard}.

On the other hand, profile decomposition for sequences of data bounded in the Lebesgue norm, usually called mass-critical, 
involve a larger group of gauges. A profile decomposition by Terence Tao, based on the cocompactness of Strichartz 
imbedding in Theorem~\ref{thm:Tao}, is derived directly from Theorem~\ref{hilbcc}). 
We refer to  \cite{KV}) for further details.
\begin{thm}(\cite{BahouriGerard})
Let $\varphi_{k}$, $\psi_{k}$ be bounded sequences in, respectively,
$\mathcal{D}^{1,2}(\R^{3})$ and $L^{2}(\R^{3})$, satisfying the
vanishing at infinity condition 
\[
\lim_{R\to\infty}\limsup_{k\to\infty}\int_{|x|>R}(|\nabla\varphi_{k}|^{2}+|\psi_{k}|^{2})\mathrm{d}x=0.
\]
Let $u_{k}\in C(\R_{t},\mathcal{D}^{1,2}(\R_{x}))\cap L_{\mathrm{loc}}^{5}(\R_{t},L^{10}(\R^{3})$
with $\partial_{t}u_{k}\in C(\R_{t},L^{2}(\R_{x}^{3}))$ denote the
solution (whose existence and uniqueness were established by Shatah
and Struwe \cite{SS1,SS2}) of the equation 
\[
\Box u+|u|^{4}u=0
\]
 in $\R_{t}\times\R_{x}^{3}$, satisfying the initial condition $u(x,0)=\varphi_{k}\rightharpoonup\varphi$,
$\partial_{t}u(x,0)=\psi_{k}\rightharpoonup\psi$, and let $u$ be
an analogous solution with the initial data $\varphi$,$\psi$. Then,
for a renumbered subsequence, there exist sequences $t_{k}^{(n)}\in\R_{t},$
$x_{k}^{(n)}\in\R_{x}^{3}$, $\lambda_{k}^{(n)}>0$, and functions
$W^{(n)}(x,t)$ , $|||W^{(n)}(x,t)|||<\infty$, such that
\[
u_{k}(x,t)=u(x,t)+\sum_{n=1}^{\ell}\lambda_{k}^{(n)^{1/2}}W^{(n)}(\lambda_{k}^{(n)}(t-t_{k}^{(n)}),\lambda_{k}^{(n)}(x-x_{k}^{(n)}))+w_{k}^{(\ell)}+r_{k}^{(\ell)},
\]
where
\[
\lim_{\ell\to\infty}\limsup_{k\to\infty}\|w_{k}^{(\ell)}\|_{L^{5}(\R_{t},L^{10}(\R_{x}^{3}))}=0,
\]
 and
\[
\lim_{\ell\to\infty}\limsup_{k\to\infty}|||r_{k}^{(\ell)}|||=0,
\]
where 
\[
|||r|||=\sup_{t\in\R}\left(\int_{\R^{3}}(|\partial_{t}r(t,x)|^{2}+|\nabla_{x}r(t,x)|^{2})\mathrm{d}x\right)^{1/2}+\|r\|_{L^{5}(\R_{t},L^{10}(\R_{x}^{3}))}.
\]
Moreover, the elementary concentrations in the profile decomposition
are asymptotically decoupled in the sense that 
\[
\left|\log\frac{\lambda_{k}^{(m)}}{\lambda_{k}^{(n)}}\right|+(\lambda_{k}^{(m)}+\lambda_{k}^{(n)})|x_{k}^{(m)}-x_{k}^{(n)}|\to\infty\text{ whenever }m\neq n.
\]
In addition, the energy functional (the quadratic portion of $|||\cdot|||$)
is additive with regard to the terms in the decomposition.
\end{thm}
\subsection{Profile decompositions for the Moser-Trudinger imbedding.}
Profile decompositions for the Moser-Trudinger imbedding can be derived
from minor adaptations of Theorem~\ref{hilbcc} or Theorem~\ref{thm:Banach}.
For the Trudinger-Moser inequality for radial functions on the unit disk, we have 
the following result. We recall that $N'=\frac{N}{N-1}$
\begin{thm} (\cite{ATdO})
Let $u_{k}\rightharpoonup0$ in $H_{0}^{1}(B)$ or let $u_{k}\rightharpoonup0$
be a sequence of radial non-increasing functions in $W_{0}^{1,N}(B)$
for $N\ge3$. There exist sequences $s_{k}^{(n)}\to\infty$, and $w^{(n)}\in W_{0,\mathrm{rad}}^{1,N}(B)$
$n\in\N$, such that for a renumbered subsequence,
\begin{eqnarray*}
 &  & w^{(n)}=\mathrm{w-lim}\left({s_{k}^{(n)}}\right)^{-1/N'}u_{k}(r^{{s_{k}^{(n)}}}),\\
 &  & |\log(s_{k}^{(m)}/s_{k}^{(n)})|\to\infty\mbox{ for }n\neq m,\\
 &  & \sum_{n\in\N}\int_{B}|\nabla w^{(n)}|^{N}\mathrm{d}x\le\mathrm{\limsup}\int_{B}|\nabla u_{k}|^{N}\mathrm{d}x,\\
 &  & u_{k}-\sum_{n\in\N}\left({s_{k}^{(n)}}\right)^{1/N'}w^{(n)}(r^{{-s_{k}^{(n)}}})\to0 \text{ in } \exp L^{N'},
\end{eqnarray*}
and the series $\sum_{n\in\N}\left({s_{k}^{(n)}}\right)^{1/N'}w^{(n)}(r^{{s_{k}^{(n)}}})$
converges in $W_{0}^{1,N}(B)$ uniformly in $k$. 
\end{thm}
\begin{rem}
An immediate generalization of this result to the radial subspace of $W^{1,N}(\R^N)$ can be stated in form of the decomposition above
for $u_k-u_k(1)\in H_{0}^{1}(B)$ with vanishing of $u_k(r)$ for $r>1$. Such decomposition is presented for $N=2$ (with an 
additional assumption on the sequence and the weak remainder) 
in \cite{BMM1}, which is very similar to \cite{ATdO} and is mentioned here only on the merit of
details on vanishing of a sequence $u_k\rightharpoonup0$ in $H^1_{0,\mathrm{rad}}$ for $r>1$. A strongly vanishing remainder is, nonetheless,
immediate.
\end{rem}
Without the assumption of radiality, profile decompositions for the Trudinger-Moser case are known when $N=2$. 
The result for the bounded domain is contained in \cite{AT}. We quote it below with a trivially refined (in view of
Corollary~3.2, \cite{ccbook}) energy estimate.
Note that the concentration profiles are always radial, even when the original
sequence $u_{k}$ is not. Indeed, when the concentration profiles
are defined as weak limits of
$j_k^{-1/2}u_k(z^{j_k}$ with $j_k\to\infty$, it is obvious that they will be
symmetric with respect to discrete rotations by an arbitrarily small angle, i. e.  radially symmetric.
\begin{thm}
\label{thm:AdiTi}
Let $\Omega\subset\R^{2}$ be a bounded domain and let $u_{k}\rightharpoonup0$
in $H_{0}^{1}(\Omega)$. There exist sequences $j_{k}^{(n)}\in\N$,
$j_{k}^{(n)}\to\infty,$ $z_{k}^{(n)}\in\Omega$, and $w^{(n)}\in W_{0}^{1,N}(B)$
$n\in\N$, such that for a renumbered subsequence,
\begin{eqnarray*}
 &  & w^{(n)}(|z|)=\mathrm{w-lim}\left({j_{k}^{(n)}}\right)^{-1/2}u_{k}(z_{k}^{(n)}+z^{j_{k}^{(n)}}),\\
 &  & z_{m}\neq z_{n}\mbox{ or }|\log j_{k}^{(m)}-\log j_{k}^{(n)}|\to\infty\mbox{ whenever }n\neq m,\\
 &  & \|\nabla u_{k}\|_2^{2}=\sum_{n\in\N}\int_{B}|\nabla w^{(n)}|^{2}dx+\|\nabla r_k\|_2^2+o(1)\\
 \text{ where} &&\\
 &  &r_k:=u_{k}-\sum_{n\in\N}{j_{k}^{(n)}}^{1/2}w^{(n)}(|z-z_{n}|^{1/j_{k}^{(n)}})\cw0,
\end{eqnarray*}
(with the latter convergence equivalent, for bounded sequences in
$H_{0}^{1}$ to convergence in $\exp L^{2})$, and the series $\sum_{n\in\N}{j_{k}^{(n)}}^{1/2}w^{(n)}(|z-z_{n}|^{1/j_{k}^{(n)}})$
converges in $H_{0}^{1}$ uniformly in $k$. 
\end{thm}
Recently, Bahouri, Majdoub and Masmoudi \cite{BMM} announced the following result, partly 
extending Theorem~\ref{thm:AdiTi} to $\Omega=\mathbb{R}^{2}$.
\begin{thm}
\label{BMM} Let $u_k\rightharpoonup 0$ in
$H^1_{\mathrm{rad}}(\R^2)$ such that 
\begin{equation} 
\limsup_{k\to\infty}\|u_k\|_{\exp_1 L^2}>0, \quad \quad
\mbox{and} \end{equation} 
\begin{equation}  \lim_{R\to\infty}
\limsup_{k\to\infty}\|u_k\|_{L^2(|x|>R)}=0. \end{equation} Then, there
exists a sequence of asymptotically 
orthogonal rescalings $s_k^{(n)},y^{(n)}$, and a sequence of profiles $w^{(n)}$ such that, 
up to a subsequence extraction, for all
$\ell\geq 1$, the following asymptotic relation holds true:
\begin{equation} \label{decomp}
u_n(x)=\sum_{n=1}^{\ell}\,{s_k^{(n)}}^{-1/2}w^{(n)}(|x-y^{(n)}|^{s^{(n)}_k})+{\rm
r}_n^{(\ell)}(x),\quad\limsup_{k\to\infty}\;\|{\rm
r}_k^{(\ell)}\|_{\exp_1 L^2}\stackrel{\ell\to\infty}\longrightarrow 0. \end{equation} 
Moreover, 
\begin{equation} \label{ortogonal} \|\nabla
u_k\|_{L^2}^2=\sum_{j=1}^{\ell}\,\|{w^{(n)}}'\|_{L^2}^2+\|\nabla
{\rm r}_n^{(\ell)}\|_{L^2}^2+\circ(1),\quad k\to\infty. \end{equation}
\end{thm}
In addition to inhomogeneous dilations $u\mapsto j^{-1/2}u(z^j)$, $j\in\Z$ of the unit disk, the gradient norm $\|\nabla u\|_2$
on the unit disk is preserved by M\"obius transformations, $u(z)\mapsto u(\frac{z-\zeta}{1-\zeta\bar z})$, $\zeta\in B$. 
This has a geometric meaning of isometries of the hyperbolic plane represented by the Poincar\'e  disk coordinates, where $\|\nabla u\|^2_2$
represents the quadratic form of the Laplace-Beltrami operator. 
This puts concentration compactness relative to such ``translations'' into the framework of Sobolev spaces on periodic manifolds, discussed
in Section 2 (and based on Lemma~9.4, \cite{ccbook}), with the profile decomposition given by Theorem~\ref{hilbcc}. 
Note only that application of Lemma~9.4, \cite{ccbook}, gives cocompactness of the imbedding of $H_0^1(B)\hookrightarrow L^q(B,\mu)$, $q\in(2,\infty)$,
where $d\mu=\frac{4dx}{(1-r^2)^2}$ is the Riemannian measure of the hyperbolic plane in the disk coordinates, 
and that equivalent metrizations of $D$-weak convergence include not only $L^q$ norms, but any Orlicz norms $\psi L(B,\mu)$ associated with the 
even convex functions $\psi$ such that $\psi(t)/t^2\to 0$ when $t\to 0$ and $\log\psi(t)/t^2\to 0$ when $t\to\infty$. 
See \cite{ATsub} for details.
\section{Cocompactness and profile decompositions by reduction}
Condition of invariance with respect to a non-compact group is of exceptional kind, and so it is important to 
consider the implications of cocompactness on spaces without such invariance.
In many cases it is advantageous to see a Banach space $X_0$ as a subspace of a larger space $X$,
that admits a profile decomposition. This profile decomposition, reduced to sequences in $X_0$,
may take a significantly simplified form, as we see from the examples below. Furthermore, there are situations
when restriction of a cocompact imbedding to a subspace results in a compact imbedding.
\par
\begin{ex}
Consider the limit Sobolev imbedding $\mathcal{D}^{1,p}(\R^{N})\hookrightarrow L^{\frac{pN}{N-p}}(\R^{N})$
and let $u_{k}$ be a sequence bounded in the $W^{1,p}(\R^{N})$-norm.
For any such sequence the $L^{p}$-bound implies that $t_{k}^{\frac{N-p}{p}}u_{k}(t\cdot)\rightharpoonup0$
whenever with $t_{k}\to0$, and consequently, a renamed subsequence
of $u_{k}$ can be written as a sum of translations $w^{(n)}(\cdot-y_{k}^{(n)})$
plus a sum of dilations (at variable cores $y_{k}$) with $t_{k}\to\infty$,
which, in turn vanishes in $L^{q}$, $q\in(p,\frac{pN}{N-p})$, plus a remainder
that vanishes in $L^{\frac{pN}{N-p}}$ (and is still bounded in $L^{p}$, and thus also vanishes in $L^{q}$).
\par
In other words, by considering
$W^{1,p}$ as a subspace of $\mathcal{D}^{1,p}$, we derived a profile decomposition in $W^{1,p}$, with translations as only gauges
and with a remainder vanishing in $L^q$, $q\in(p,\frac{pN}{N-p})$.
\end{ex}
\subsection{Transitivity and interpolation}
Cocompactness of imbeddings is often possible to infer by means of
the following elementary arguments. 
\textbf{A.} Let $D$ be a set of isometries on three nested Banach
spaces: $X\hookrightarrow Y$ and $Y\hookrightarrow Z$, and one of
the two imbeddings is cocompact. Let $g_{k}u_{k}\rightharpoonup0$
in $X$ for any $g_{k}\in D$. If the first imbedding is cocompact,
then $u_{k}\to0$ in $Y$ and thus $u_{k}\to0$ in $Z$. If the
second imbedding is cocompact, then, by continuity of the first imbedding,
$g_{k}u_{k}\rightharpoonup0$ in $Y$ and thus $u_{k}\to 0$ in
$Z$.

\textbf{B.} Let $X\hookrightarrow Y_{0}$, $X\hookrightarrow Y_{1}$,
and assume that the first imbedding is cocompact. If the convergence in
$Y_{0}$ and $Y_{1}$, for sequences bounded in $X$, is equivalent, then
obviously, the second imbedding is cocompact as well. 

\textbf{C.} Let $X\hookrightarrow Y_{0}$, $X\hookrightarrow Y_{\alpha}$,
and assume that the first imbedding is cocompact. If for any $u\in X$,
$\|u\|_{Y_{\alpha}}\le C\|u\|_{Y_{0}}^{\alpha}\|u\|_{X}^{1-\alpha}$
with some $\alpha\in(0,1)$, then $X$ is cocompactly imbedded
into $Y_{\alpha}$. Indeed, $ $the inequality implies that $Y_{0}\cap B^{X}\hookrightarrow Y_{\alpha}\cap B^{X}$.

In particular, two known proofs of cocompactness of the Sobolev imbedding
$\mathcal{D}^{1,p}(\mathbb{R}^{N})\hookrightarrow L^{p^{*}}$ use
the argument \textbf{C}, by first establishing cocompactness of imbedding
of $\mathcal{D}^{1,p}$ into a Besov space $\dot{B}^{1-N/p,\infty;\infty}$
(Jaffard \cite{Jaffard}) or a Marcinkiewicz space $L^{\frac{pN}{N-p},\infty}$ (Solimini \cite{Solimini}).
\vskip5mm\par
In general, one would also expect that, given a common set of gauges,
interpolation of two imbeddings, $X_{0}\hookrightarrow Y_{0}$, and
$X_{1}\hookrightarrow Y_{1}$, results in a cocompact imbedding, if
one of this imbeddings is cocompact. We refer to \cite{CwiTi}
where a more specific statement is proved, under additional conditions, for
functional spaces of $\R^{N}$, which is then applied to verify that
subcritical imbeddings of Besov spaces
\begin{description}
\item [{$B_{p,q}^{s}(\R^{N})\hookrightarrow B_{p_{1},q_{1}}^{s_{1}}(\R^{N})$,}] $q_{1}\ge q$,
$0<\frac{1}{p}-\frac{1}{p_{1}}<\frac{s-s_{1}}{N}$, $s_{1}\ge0$. 
\end{description}
are cocompact with respect to lattice shifts $D=\lbrace u\mapsto u(\cdot-y)\rbrace_{y\in\Z^{N}}$.
This result can be also deduced from cocompactness of homogeneous
Besov spaces mentioned in Section 2, by means of the reduction method
below.

\subsection{Reduction to a subspace}

\begin{ex}
Reduction of Solimini's profile decomposition for $\mathcal{D}^{1,p}(\R^{N})\hookrightarrow L^\frac{pN}{N-p}$
to the subspace of radial functions eliminates from the profile decomposition all translations
and leaves only the concentrations with dilations about the origin. This can be also obtained 
directly from Proposition~\ref{prop:radial} and Theorem~\ref{thm:Banach} above.
\end{ex}
\begin{ex} If $\Omega\subset\R^{N}$
is a  bounded domain, one obtains a profile decomposition 
for the imbedding $W_{0}^{1,p}(\Omega)\hookrightarrow L^{\frac{pN}{N-p}}(\Omega)$
by reducing Solimini's decomposition as follows. By Friedrichs
inequality, the sequence $u_k$ has a $L^{p}$ -bound, which and eliminates all profiles subjected to unbounded deflations.
Furthermore, there are no concentrations with $t_{k}\to\infty$ at cores lying outside of $\Omega$, since the corresponding weak
limits will necessarily be equal zero. For the same reason, if $t_{k}$
if bounded (equivalently, with $t_{k}=1)$, there are no concentration
terms with translations by unbounded sequences $y_{k}$. As a result,
every bounded sequence in $W_{0}^{1,p}(\Omega)$ has a subsequence
consisting of a countable sum of local concentrations $t_{k}^{\frac{N-p}{p}}w(t_{k}(\cdot-y)),$
$t_{k}\to\infty$, and a remainder vanishing in $L^{\frac{pN}{N-p}}$.
This local concentration can be easily transferred, using the exponential map, to compact Riemannian manifolds, giving
rise to profile decompositions, which were introduced, under the name of global compactness by Struwe \cite{Struwe84} for Palais-Smale sequences,
(see Theorem~3.1 in \cite{DruHeRo}), although, allowing infinitely many terms and arbitrary profiles, 
they can be easily re-established for general sequences. 
\end{ex}
\subsection{Compactness as reduced cocompactness}
Reduction to subspaces may in some cases eliminate all concentrations, which
makes a restriction of the cocompact imbedding to a subspace a compact imbedding. 
In this case the argument does not involve profile decompositions and uses only cocompactness. 
We give here two examples: subspaces defined by restriction of support (compactness of Sobolev imbeddings
on domains thin at infinity) and subspaces defined by a compact symmetry. The following statement is simpler
and more general than its partial counterparts in \cite{ccbook}, and we bring it with a proof.
\begin{prop}
Let $M$ be a complete Riemannian $N$-manifold, periodic
relative to some subgroup $G$ of its isometries.
Let $X$ be a reflexive Banach space cocompactly imbedded into $L^{q}(M)$ for some $1<q<\infty$,
and assume that $\|u\circ\eta\|_{q}=\|u\|_{q}$ for all $\eta\in G$.
Assume that the imbedding $X\hookrightarrow L^q(M)$ is cocompact relative to the action of $G$.
Let $\Omega\subset M$
be an open set such that the measure of the set $\liminf\eta_{k}\Omega$, $\eta_{k}\in G$,
is zero whenever $\eta_k x_0$ has no convergent subsequence for some $x_0\in M$.
If $X(\Omega)$ is the subspace of $X$ consisting
of functions that equal zero a.e. on $M\setminus\Omega$, then the
imbedding $X(\Omega)\hookrightarrow L^{q}(\Omega)$ is compact.
 \end{prop}
\begin{proof}
Let $u_{k}\in X(\Omega)$ be a bounded sequence, and assume
without loss of generality, that $u_{k}\rightharpoonup0$. By assumption,
if $\eta_{k}x_0$ has no convergent subsequence, then any weakly convergent
subsequence of $u_{k}\circ\eta_{k}$ converges weakly to a function, supported, as a $L^{q}$-function, 
on a set of measure zero, i.e., by the imbedding, to the zero element of $X$. Assume now that, on a
renumbered subsequence, $\eta_{k}\to\eta$. Then 
$\mathrm{w-lim}\, u_{k}\circ\eta_{k}=\mathrm{w-lim}\, u_{k}\circ\eta=0$ in $X$. 
Since this is true for any convergent subsequence, this is true for the original $\eta_{k}$.
We conclude then that in any case $u_{k}\circ\eta_{k}\rightharpoonup 0$,
and thus, by cocompactness of the imbedding, $u_{k}\to0$ in $L^{q}(M)$.
\end{proof}
The following theorem includes as the particular cases, the Strauss lemma \cite{Strauss}
as well as its generalization to subspaces with the block-radial symmetry.
\begin{thm} (Compactness under coersive symmetries, \cite{SkrTin})
Let $M$ be a complete Riemannian $N$-manifold
with a transitive group $G$ of isometries. Let $X$ be a Banach space
cocompactly imbedded into $L^{q}$ for some $q\in(1,\infty)$, relative
to the actions of $G$. Let $\Omega\subset G$ be a compact subgroup
and let $X_{\Omega}$ be a subspace of $X$ invariant with respect
to actions of $\Omega$. The imbedding $X_{\Omega}\hookrightarrow L^{q}(M)$
is compact if and only if for any $t>0$ the set 
\[
\lbrace x\in M:\;\mathrm{diam}(\Omega x)\le t\rbrace
\]
is bounded.
\end{thm}
\subsection{Flask spaces}
Profile decomposition can be naturally extended to a class of subspaces
that are not gauge-invariant. Del Pino and Felmer
\cite{PinoFelmer} have discovered them as Sobolev spaces of ``flask domains'' in $\R^{N}$. 
On the functional-analytic level flask spaces were defined in \cite{ccbook}. 
\begin{defn}
Let $D$ be a set of surjective isometric operators on
a Banach space $X$. A subspace $X_{0}\subset X$ is called a $D$-flask
subspace if the set of sequential weak limits $\mathrm{w-lim}(DX_{0})$
of sequences $\lbrace g_{k}u_k,\; g_{k}\in D,u_{k}\in X_{0}\rbrace_{k\in\N}$
remains in $DX_{0}$.
 \end{defn}
This property does not hold, in particular, if $X\hookrightarrow L^{q}(\R^{N})$,
$D$ includes dilations, and $X_0\subsetneq X$ contains a continuous function, or if $D$ includes all translations
and for any $R>0$ the subspace $X_0\subsetneq X$, contains a function $f_R$ positive on some ball of radius $R$.
\par
On the other hand if $X(\Omega)=W_{0}^{1,p}(\Omega)$
where $\Omega=(-1,1)\times\R^{N-1}\cup\R^{N-1}\times(-1,1)$, an infinite
cross, and $D$ is a group of shifts by $\Z^{N}$, then any nonzero
translated weak limit of a sequence from $W_{0}^{1,p}(\Omega)$ will
be supported on a domain that is a translate of $\Omega$, or a translate
of $(-1,1)\times\R^{N-1}$, or a translate of $\R^{N-1}\times(-1,1)$,
which in any case is a subset of $\Omega$. We have an immediate statement
which is an elementary generalization of Proposition~3.5 of \cite{ccbook}.
\begin{thm}
Assume that $D$ be a group of surjective isometric operators
on a Banach space $X$, such that the profile decomposition \eqref{pd} holds true. 
If $X_{0}\subset X$ is a $D$-flask subspace, then any bounded sequence in $X_0$ has a 
profile decomposition \eqref{pd} with profiles $w^{(n)}\in X_{0}$. 
\end{thm}
\begin{proof}
Since $X_{0}$ is a $D$-flask space, for each $n\in\N$ there
exists $g_{n}\in D$ such that $\tilde{w}^{(n)}=g_{n}w^{(n)}\in X_{0}$.
The profile decomposition \eqref{pd} can be rewritten then with $\tilde{w}^{(n)}$
replaced with $w^{(n)}$ and $g_{k}^{(n)}$ replaced by $\tilde{g}_{k}^{(n)}=g_{k}^{(n)}g_{n}^{-1}$. 
\end{proof}
The following sufficient condition for a $D$-flask set is a slightly generalized version of Remark~9.1(d)
in \cite{ccbook}.
\begin{thm}
Let $M$ be a complete Riemannian manifold, and let $G$ be
a subgroup of its isometries. Let $X$ be a reflexive Banach space
continuously imbedded into $L^{q}(M)$, $1<q<\infty$, let $D=\lbrace u\mapsto u\circ\eta\rbrace_{\eta\in G}$
and assume that $\|u\circ\eta\|_{X}=\|u\|_{X}$ for all $\eta\in G$.
Let $X(\Omega)$, $\Omega\subset M,$ be a subspace of all functions
in $X$ that vanish a.e. in $M\setminus\Omega$. If for any sequence
$\eta_{k}\in M$ there exist $\eta\in G$ and a set of zero measure
$Z\subset M$, such that $\liminf\eta_{k}\Omega\subset\eta\Omega\cup Z$,
then $X(\Omega)$ is a $D$-flask set. 
\end{thm}

\end{document}